\documentclass[12pt]{amsart}
\usepackage[applemac]{inputenc}
\usepackage{times, ulem, amsfonts}
\usepackage{mathpazo, amsmath, amssymb, amsthm,color}

\date{}
\definecolor{sah}{rgb}{0.66,0.33, 0.04}
\definecolor{adel4}{cmyk}{1,0,0,0}
\definecolor{adel3}{rgb}{0.66,0.33, 0.04}
\definecolor{adel1}{cmyk}{0,0.20,1,0}
\definecolor{adel2}{cmyk}{0,0.40,1,0.30}
\definecolor{adel0}{rgb}{0.99,0.60, 0.30}
\definecolor{trut}{rgb}{0.99,0.80, 0.00}
\definecolor{trus}{rgb}{0.00, 0.50, 0.00}
 \definecolor{trust}{rgb}{0.99, 0.99, 0.80}
\definecolor{MaCouleur}{rgb}{0,0.9,0.3}
\def\virgp{\raise 2pt\hbox{,}}
\def\({\left(}
\def\){\right)}

\def\<{\langle}
\def\>{\rangle}

\theoremstyle{plain}

\newtheorem*{ex}{Example}

\newtheorem{Theo}{Theorem}
 \newtheorem{Lemm}{Lemma}

\newtheorem{Prop}{Proposition}
 
 \newtheorem{Coro}{Corollary}
 \newtheorem{Defin}{ Definition}
 
 \newtheorem{rema}{Remark}

\newcommand{\av}{{\rm Avg}}

\newcommand{\R}{{\mathbb R}}

\def\be#1 {\begin{equation} \label{#1}}
\newcommand{\ee}{\end{equation}}

\newcommand{\mb}{\medskip\noindent}

\newcommand{\LL}{{\mathcal L}}

\usepackage[textmath,displaymath,floats,graphics,delayed]{preview}
  \title[Sharp constants for composition with a measure-preserving map]{Sharp constants for composition with a bi-Lipschitz measure-preserving map}
\author[F. Bernicot]{Fr\'ed\'eric Bernicot}
\address{CNRS - Universit\'e de Nantes \\ Laboratoire de Math\`ematiques Jean Leray \\ 2, Rue de la Houssini\`ere F-44322 Nantes Cedex 03, France}
 \email{frederic.bernicot@univ-nantes.fr}
\author[S. Keraani]{Sahbi Keraani}
\address{UFR de math\'ematiques\\ Universit\'e de Lille 1\\
 59655 Villeneuve d'Ascq  Cedex\\   France}
\email{sahbi.keraani@univ-lille1.fr}

\thanks{The two authors are supported by the ANR under the project AFoMEN no. 2011-JS01-001-01.}

\date{\today}


\begin{document}

\begin{abstract} In this note, we aim to describe sharp constants for the composition operator with a bi-Lipschitz measure-preserving map in several functional spaces (BMO, Hardy space, Carleson measures, ...). It is interesting to see how the measure preserving property allows us to improve these constants. Moreover, we will prove the optimality of our results for the BMO space and describe improved estimates for solutions of transport PDEs.
\end{abstract}
\maketitle

Denote by \mbox{$\LL:=\LL(\R^d)$} the group of all bi-Lipschitz homeomorphisms of \mbox{$\R^d$} (equipped with the composition law \mbox{$\circ$}) and for \mbox{$\phi\in\LL$} let
$$ K(\phi)=K_\phi := \sup_{x\neq y} \frac{|\phi(x)-\phi(y)|}{|x-y|}+\frac{|x-y|}{|\phi(x)-\phi(y)|}$$ 
denote the sum of the Lipschitz constants of \mbox{$\phi$} and \mbox{$\phi^{-1}$}. An easy computation yields that \mbox{$K(\phi)\geq 2$}, with equality if and only if \mbox{$\phi$} is an isometry of \mbox{$\R^d$}, and moreover this quantity is sub-multiplicative: for \mbox{$\phi,\psi\in \LL$} then
$$K(\phi \circ \psi) \leq K(\phi) K(\psi).$$ Consequently,  \mbox{$\phi \mapsto \log (K(\phi))$} is a pseudo-norm on \mbox{$\LL/E(d)$}, where \mbox{$E(d)$} is the set of isometries on \mbox{$\R^d$} to \mbox{$\R^d$}, which enables us to consider the group \mbox{$\LL/E(d)$} as a topological object.

A natural question then arises: for which functional Banach space \mbox{$X$} (that is, a space of functions from \mbox{$\R^d$} to \mbox{$\R$}), is the bilinear map 
$$(f,\phi) \to f\circ \phi$$
bounded from \mbox{$X \times (\LL/E(d))$} to \mbox{$X$}? By boundedness, we mean that we have the estimate
$$ \|f\circ \phi\|_{X} \lesssim  \log(K(\phi)) \|f\|_{X}.$$

We shall answer this question when \mbox{$\phi \in \LL$} is measure preserving and \mbox{$X$} is the space of functions of bounded mean oscillation BMO (Section \ref{sec:bmolip} provides the precise definition of this space). We also consider the space of Lipschitz functions (defined, again, in Section \ref{sec:bmolip}) and the space of Carleson measures (Section \ref{sec:carl}) and their dual spaces.

For a measure preserving \mbox{$\phi \in \LL$}, the bi-Lipschitz property property implies that for a ball \mbox{$B$}, \mbox{$\phi(B)$} can be contained in a ball of radius \mbox{$K_\phi r_{B}$}, where \mbox{$r_{B}$} is the radius of \mbox{$B$}. As a consequence, \mbox{$\phi(B)$} can be covered by a collection of \mbox{$K_\phi^d$} balls, with same radius as \mbox{$B$}. This observation easily yields than for a $\text{BMO}_p$ function (the BMO space with \mbox{$L^p$} oscillation) \mbox{$f$} on \mbox{$\R^d$}
$$
 \|f \circ \phi\|_{\text{BMO}_p} \leq K_\phi^{d/p} \|f\|_{\text{BMO}_p}
$$ 
and since, by the John-Nirenberg property, \mbox{$\text{BMO}_p$} coincides with BMO, by varying \mbox{$p\in(1,\infty)$}, if follows that for all \mbox{$\epsilon>0$} we have
\begin{equation} \|f \circ \phi\|_{\text{BMO}} \lesssim_\epsilon K_\phi^{\epsilon} \|f\|_{\text{BMO}}. \label{eq:BMO} \end{equation}
The interesting point of this work is to prove that the dependency on the constant \mbox{$K_\phi$} can be improved, under the measure preservation of \mbox{$\phi$}. In fact, the optimal dependency on \mbox{$K_\phi$} is logarithmic.

Similarly, it follows that for a Carleson measure \mbox{$\mu$} on \mbox{$\R^d$}, the pull-back measure
$$ \mu^ {\sharp \phi} = \mu( Id \otimes \phi^ {-1})$$
is a Carleson measure with
\begin{equation} \| \mu^ {\sharp \phi}\|_{\mathcal C} \lesssim K_\phi^d  \| \mu \|_{\mathcal C}. \label{eq:measure} \end{equation}
This can also be improved to a logarithmic dependency on \mbox{$K_\phi$} in certain circumstances.

Our main results are summarized in the following theorem. 

\begin{Theo} \label{thm} Let us assume that \mbox{$\phi$} bi-Lipschitz function preserving the measure on \mbox{$\R^d$}, then
\begin{itemize}
\item[$\bullet$][\textsc{BMO functions}] there exists an implicit constant (independent of \mbox{$K_\phi$}) such that for every BMO function \mbox{$f$}
\begin{equation} \|f \circ \phi\|_{\text{BMO}} \lesssim \log(K_\phi) \|f\|_{\text{BMO}}. \label{eq:BMO-bis} \end{equation}
\item[$\bullet$][\textsc{H\"older functions}] there exists an implicit constant (independent of \mbox{$K_\phi$}) such that for every function \mbox{$f\in \textrm{Lip}_p(a)$}
\begin{equation} \|f \circ \phi\|_{\text{Lip}_p(a)} \lesssim K_\phi^a \|f\|_{\text{Lip}_p(a)}. \label{eq:holder} \end{equation}
\item[$\bullet$] [\textsc{Carleson measures}] there exist a class \mbox{${\mathcal SC}$} of Carleson measures and an implicit constant (independent of \mbox{$K_\phi$}) such that for every Carleson measure \mbox{$\mu\in {\mathcal SC}$}, \mbox{$\mu^{\sharp \phi}$} belongs to \mbox{${\mathcal SC}$} and 
\begin{equation} \| \mu^ {\sharp \phi}\|_{\mathcal C} \lesssim \log(K_\phi) \| \mu \|_{\mathcal C}. \label{eq:measure-bis} \end{equation}
\end{itemize}
Moreover, we will prove that the logarithmic growth is optimal for the estimates involving the BMO norms.
\end{Theo}

As a corollary and using \mbox{$H^1-\textrm{BMO}$} duality, we have

\begin{Coro} Let us assume that \mbox{$\phi$} is a bi-Lipschitz function preserving the measure on \mbox{$\R^d$}, then there exists an implicit constant (independent of \mbox{$K_\phi$}) such that for every function belonging to the Hardy space \mbox{$f\in H^1$}
\begin{equation*} \|f \circ \phi\|_{H^1} \lesssim \log(K_\phi) \|f\|_{H^1}. \end{equation*}
That means in some sense that the image of an atom by the composition \mbox{$a \circ \phi$} can be split into the sum of \mbox{$\log(K_\phi)$} atoms.
\end{Coro}

One of the main motivation is the study of transport PDEs, associated to a free-divergence vector field. Indeed, such a vector field gives rise to a bi-Lipschitz measure preserving flow, which plays a crucial role for solving the transport equation. We also describe some consequences (we obtain an improved growth of the solution) in the last section for such PDEs. We point out that such study have already been done for Besov spaces, see \cite[Thm 4.2]{V} where Vishik obtained a logarithmic growth (as our result for BMO space) for the Besov space \mbox{$B^0_{\infty,1}$} with applications to Euler equation. More recently, the authors have used similar ideas in \cite{BK} to get well-posedness results for Euler equation, with a vorticity belonging to a space strictly imbricated between \mbox{$L^\infty$} and BMO. In these two results, the spaces are of completely different nature but the same idea is to understand and to have sharp inequalities for the composition (by a measure-preserving map) in these spaces.

\section{A geometric Lemma}

Before proving Theorem \ref{thm}, we would like to point out the key argument: a geometric lemma, which describes how a ball is modified by a measure-preserving map.

For each fixed \mbox{$a \in [0,1]$}, define \mbox{$\rho_a \colon [0,\infty) \to \R$} as
\begin{equation} \label{fnc}
\rho_a(r) =  \begin{cases}
   r^a,        & \text{if \mbox{$a > 0$};} \\
   \log(r),        & \text{if \mbox{$a=0$}.}
  \end{cases}
\end{equation}

\begin{Lemm} \label{lemma} For every ball \mbox{$B=B(x_0,r)$} in \mbox{$\R^d$}, there exists a collection \mbox{$(O_k)_k$} of balls such that
\begin{itemize}
 \item The collection \mbox{$(2O_k)_k$} is a bounded covering of \mbox{$\phi(B)$}
 \item The collection \mbox{$(O_k)_k$} is disjoint
 \item By writing \mbox{$r_{O_k}$} the radius of \mbox{$O_k$}, then for all \mbox{$p\in [1,\infty)$}
  \begin{equation} \left(\frac{1}{|B|} \sum_k |O_k| \rho_a(r_B/r_{O_k})^p \right)^{1/p}  \lesssim \rho_a\left(K_\phi\right), \label{eq:covering} \end{equation}
  with an implicit constant dependent only on the dimension \mbox{$n$}, \mbox{$a$} and \mbox{$p$}. 
\end{itemize}
\end{Lemm}

\begin{proof}
Let us consider a Whitney covering of the open set \mbox{$\phi(B)$}: that is a collection of open balls \mbox{$(O_k)_k$} such that~:
 \begin{itemize}
 \item the collection of double balls is a bounded covering~:
 $$ \phi(B) \subset \cup_k 2 O_k$$
  \item the collection is disjoint and for all \mbox{$k$}, \mbox{$O_k\subset \phi(B)$}
  \item the Whitney property is verified:
  $$ r_{O_k} \simeq d(O_k , \phi(B) ^c).$$
  \end{itemize}
So it remains for us to check \eqref{eq:covering}. Indeed, this is a combinatorial argument. 
First, since \mbox{$\phi$} is measure preserving, it follows that \mbox{$|O_k|\leq |B|$} and so \mbox{$r_{O_k}\leq r_B$} for all \mbox{$k$}. 
For a nonnegative integer \mbox{$l\geq 0$}, we write
$$ u_l := \sum_{k,\ 2^{-l} r_B \leq r_{O_k}< 2^{-l+1} r_B } \  \ |O_k|.$$
 Since \mbox{$(2O_k)_k$} is a bounded covering of \mbox{$\phi(B)$} (and that the balls \mbox{$(O_k)$} are disjoint), we have
 \be{eq:1} \sum_{l} u_l  \simeq |\phi^{-1}(B)|=|B|.\ee
Moreover, we see that
 \begin{align*}
 (\ast):=\left(\frac{1}{|B|} \sum_k |O_k| \rho_a(r_B/r_{O_k})^p \right)^{1/p}  & \leq \left( \frac{1}{|B|} \left[\sum_{l\geq 0} \rho_a(2^l)^p u_l\right] \right)^{1/p}.
 \end{align*}
 However for all the balls \mbox{$O_k$} considered in the term \mbox{$u_l$}, we know that \mbox{$\phi(O_k)$} are disjoint and are contained in a domain a distant at most \mbox{$K_\phi 2^{-l} r_B$} from the boundary of \mbox{$B$} (because of the Whitney property and the Lipschitz regularity of \mbox{$\phi$}). We also deduce, since the measure is preserved, that
 $$ u_l \lesssim K_\phi 2^{-l} r_B^d = K_\phi 2^{-l}|B|.$$
Therefore
 $$ u_l \lesssim \min\{K_\phi 2^{-l}|B|,u_l\}.$$
We choose an integer \mbox{$k_0\geq 1$} such that \mbox{$K_\phi 2^{-k_0} \simeq 1$} and we compute the sum as follows:
 \begin{align*}
 (\ast) & \leq \left(\frac{1}{|B|} \left[\sum_{l\geq 0} \rho_a(2^l)^p \min\{K_\phi 2^{-l}|B|,u_l\} \right] \right)^{1/p}  \\
 & \leq \left( \frac{1}{|B|} \left[\sum_{l=0}^{k_0} \rho_a(2^l)^p u_l\right] \right)^{1/p} + \left(\frac{1}{|B|} \left[\sum_{l\geq k_0} K_\phi \rho_a(2^l)^p 2^{-l} |B| \right]\right)^{1/p} \\
 & \lesssim \rho_a(2^{k_0}) + (K_\phi 2^{-k_0})^{1/p} \rho_a(2^{k_0}) \lesssim \rho_a(2^{k_0}).
 \end{align*}
 We have used \eqref{eq:1} to estimate the first sum, and to estimate the second sum we have used \mbox{$k_0\geq 1$} with  
 $$\sum_{l\geq k_0} 2^{-l} l^p = 2^{-k_0} k_0^p \sum_{j\geq 0} 2^{-j} \left(1+k_0^{-1} j\right)^p \lesssim 2^{-k_0} k_0^p$$
in the case \mbox{$a=0$} and 
 $$\sum_{l\geq k_0} 2^{(ap-1)l} = 2^{(ap-1)k_0} \sum_{j\geq 0} 2^{(ap-1)j} \lesssim 2^{-k_0} 2^{k_0ap}$$
in the case \mbox{$a>0$}. This concludes the proof of \eqref{eq:covering} since \mbox{$\rho_a(2^{k_0}) \lesssim \rho_a(K_\phi)$} . 
\end{proof}

\section{The behavior of BMO functions and Carleson measures}

\subsection{BMO and Lipschitz functions} \label{sec:bmolip}

For \mbox{$p\in[1,\infty)$}, \mbox{$a\in[0,1]$} and a locally integrable function \mbox{$f$} set
\[
\|f\|^\sharp_{p,a} = \sup_{B} \left(\frac{1}{|B|^{1+ap/d}} \int_B \left|f(y) - \av_{B} (f)\right|^p dy\right)^{1/p}
\]
where the suprema are taken over all balls \mbox{$B \subset \R^d$}. As usual, we define \mbox{$\text{BMO}_p := \{f\in L^1_{\text{loc}} \, | \, \|f\|^\sharp_{p,0} < \infty\}$} and \mbox{$\text{Lip}_p(a) := \{f\in L^1_{\text{loc}} \, | \, \|f\|^\sharp_{p,a} < \infty\}$} for \mbox{$a \in (0,1]$}. The John-Nirenberg property shows us that all the spaces \mbox{$\text{BMO}_p$ ($1\leq p < \infty$)} coincide, so we write \mbox{$\text{BMO}_1 = \text{BMO}$ and $\|f\|^\sharp_{p,0} = \|f\|_{\text{BMO}}$}. The main result of this subsection is the following one:

\begin{Theo} \label{thmb}
Let \mbox{$p\in[1,\infty)$}, \mbox{$a\in[0,1)$} and \mbox{$\phi:\R^d \rightarrow \R^n$} be a bi-Lipschitz measure-preserving map. Then 
$$ \|f\circ\phi\|^\sharp_{p,a} \lesssim \rho_a(K_\phi)\|f\|^\sharp_{p,a}$$
where \mbox{$\rho_a$} is defined in \eqref{fnc}.
\end{Theo}

\begin{rema} In the light of the John-Nirenberg inequality, it is not surprising that the required ``extra'' factor \mbox{$\rho_a(K_\phi)$} does not depend on the exponent \mbox{$p\in[1,\infty)$}.
 \end{rema}

\mb We will require the following well-known lemma (at least for \mbox{$BMO$}):

\begin{Lemm} \label{lem} Let \mbox{$f\in BMO_p$} and \mbox{$B$} be a ball in \mbox{$\R^d$} which contains \mbox{$x \in \R^d$}, then for all \mbox{$\lambda>1$}
$$ \left| \av_{B}(f) - \av_{\lambda B} (f) \right| \lesssim \rho_a(\lambda)|B|^{a/d} \|f\|^\sharp_{p,a}.$$
\end{Lemm}

\mb For readibility and the sake of completeness, we will provide the proof.

\begin{proof} First we remark that, using the doubling property of the Euclidean measure,
\begin{align*}
 \left|\av_{B}(f) - \av_{2B} (f) \right| & \leq \av_{B} |f-\av_{2B}(f)| \lesssim \av_{2B} |f-\av_{2B}(f)| \\ 
 & \lesssim \left(\av_{2B} |f-\av_{2B}(f)|^p \right)^{1/p}. 
\end{align*}
So it follows that 
$$ \left| \av_{B}(f) - \av_{2B} (f) \right| \lesssim 2^a|B|^{a/n}\|f\|^\sharp_{p,a},$$
which corresponds to the desired result for \mbox{$\lambda=2$}.

We iterate this argument \mbox{$k_0$} times, where \mbox{$k_0$} is such that \mbox{$2^{k_0} \leq \lambda <2^{k_0+1}$}, and obtain
\begin{align*}
\left| \av_{B} (f) - \av_{2^{k_0} B} (f) \right| & \leq \sum_{k=0}^{k_0-1} \left| \av_{2^k B}(f) - \av_{2^{k+1} B} (f) \right| \\
 & \lesssim \sum_{k=0}^{k_0-1} 2^{a(1+k)}|B|^{a/d}\|f\|^\sharp_{p,a} \lesssim \rho(2^{k_0})|B|^{a/d} \|f\|^\sharp_{p,a}.
\end{align*}
To conclude, it remains for us to estimate the following term:
\begin{align*}
 \left| \av_{2^{k_0} B} (f) - \av_{\lambda B} (f) \right| & \lesssim \av_{2^{k_0} B} \left|f- \av_{\lambda B} (f) \right| \\
 & \lesssim \av_{\lambda B} \left|f- \av_{\lambda B} (f) \right| \\
 & \lesssim \lambda^a|B|^{a/d}\|f\|^\sharp_{p,a} \lesssim \rho_a(\lambda)|B|^{a/d}\|f\|^\sharp_{p,a},
\end{align*}
where we have one more time used the doubling property.
\end{proof}
\mb We can now prove Theorem \ref{thmb}.

\begin{proof}[Proof of Theorem \ref{thmb}]
Fix \mbox{$x \in \R^d$} and let \mbox{$B$} be a ball in \mbox{$\R^d$} containing \mbox{$x$}. We wish to estimate
$$ I:=\left(\av_{B} \left| f \circ \phi - \av_{B} (f\circ \phi) \right|^p \right)^{1/p}$$
and with this in mind, using the measure-preserving property of \mbox{$\phi$}, we see that
\begin{align*}
I = \left(\frac{1}{|B|}\int_{\phi(B)}\left| f(z)- \av_{\phi(B)}(f) \right|^p dz\right)^{1/p}. 
\end{align*}
 Now we would like to compare \mbox{$\phi(B)$} with the ball
 $$ \tilde{B}:= B(\phi(x),r).$$
 We claim that we have the following inequality
 \be{amontrer}
 II:=\left(\av_{\phi(B)} \left| f - \av_{\tilde{B}} (f)\right|^p \right)^{1/p} \lesssim \rho_a(K_\phi) |B|^{a/d} \|f\|^\sharp_{p,a}.
 \ee
Let us first deduce the theorem from \eqref{amontrer}. \\
We have
 $$ \left| \av_{\phi(B)} (f) - \av_{\tilde{B}}(f) \right| \lesssim \rho_a(K_\phi) |B|^{a/d} \|f\|^\sharp_{p,a}, $$
hence
$$ I \lesssim \rho_a(K_\phi) |B|^{a/d} \|f\|^\sharp_{p,a},$$
where we have used \eqref{amontrer} a second time and so we deduce the desired result. \\
It remains for us to prove \eqref{amontrer}. To achieve this we use the collection \mbox{$(O_k)_k$} given by Lemma \ref{lemma}:
\begin{itemize}
 \item The collection \mbox{$(2O_k)_k$} is a bounded covering of \mbox{$\phi(B)$}
 \item The collection \mbox{$(O_k)_k$} is disjoint
 \item By writing \mbox{$r_{O_k}$} the radius of \mbox{$O_k$}, then for all \mbox{$p\in [1,\infty)$}
  \begin{equation} \left(\frac{1}{|B|} \sum_k |O_k| \rho_a(r_B/r_{O_k})^p \right)^{1/p}  \lesssim \rho_a\left(K_\phi\right). \label{eq:covering2} \end{equation}
  \end{itemize}
So we have
  $$ 	II = \left(\frac{1}{|B|} \sum_k \int_{2O_k} \left| f - \av_{\tilde{B}} (f)\right|^p \right)^{1/p}.$$
Let us first remark that since \mbox{$\phi$} is measure preserving then \mbox{$|O_k|\leq |\phi(B)|=|B|$} so \mbox{$r_{O_k}\leq r_B$}.
As a consequence, we see that \mbox{$2O_k \subset 2K_\phi \tilde{B}$} and so 
  \begin{align*}
  \left| \av_{2O_k} (f)  - \av_{\tilde{B}} (f)  \right| & \leq \left| \av_{2O_k} (f) - \av_{2K_\phi\tilde{B}} (f)  \right|+ \left| \av_{2K_\phi\tilde{B}} (f) - \av_{\tilde{B}} (f)  \right|  \\
  & \lesssim \left| \av_{2O_k} (f) - \av_{2K_\phi\tilde{B}} (f)  \right| + \rho_a(2K_\phi)|B|^{a/d} \|f\|^\sharp_{p,a} \\
  & \lesssim \left| \av_{2O_k} \left(f - \av_{2K_\phi\tilde{B}} (f)\right) \right| + \rho_a(2K_\phi)|B|^{a/d} \|f\|^\sharp_{p,a} \\
  & \lesssim \left| \av_{B(x_{O_k},4K_\phi r_B)} \left(f - \av_{2O_k} (f)\right) \right| + \rho_a(2K_\phi)|B|^{a/d} \|f\|^\sharp_{p,a} \\
  & \lesssim \left| \av_{B(x_{O_k},4K_\phi r_B)} f - \av_{2O_k} (f)\right| + \rho_a(2K_\phi)|B|^{a/d} \|f\|^\sharp_{p,a} \\
  & \lesssim \left[\rho_a(2K_\phi r_B/r_{O_k})|O_k|^{a/d} + \rho_a(2K_\phi)|B|^{a/d}\right]\|f\|^\sharp_{p,a} \\
  & \lesssim \left[\rho_a(r_B/r_{O_k}) + \rho_a(K_\phi)\right]|B|^{a/d}\|f\|^\sharp_{p,a},
  \end{align*}
 where we have used Lemma \ref{lem} once and then a second time with the fact that \mbox{$B(x_{O_k},4K_\phi r_B)$} is a dilation of \mbox{$2O_k$} by a factor \mbox{$2K_\phi r_B/r_{O_k}$}. Consequently (using that \mbox{$(2O_k)_k$} is a bounded covering of \mbox{$\phi(B)$} and that the measure is preserved), we deduce  that 
  \begin{align*}
   II &\lesssim \left(\frac{1}{|B|} \sum_k \left[ \int_{2O_k} \left| f - \av_{2O_k} (f)\right|^p + |O_k| \left(\left[\rho_a(r_B/r_{O_k}) + \rho_a(K_\phi)\right]|B|^{a/d}\|f\|^\sharp_{p,a} \right)^p \right]\right)^{1/p}  \\
 & \lesssim   \left(\frac{1}{|B|} \sum_k |O_k| \big(\rho_a(r_B/r_{O_k}) + \rho_a(K_\phi)\big)^p\right)^{1/p} |B|^{a/d}\|f\|^\sharp_{p,a}   \\
 & \lesssim  \rho_a(K_\phi)|B|^{a/n}\|f\|^\sharp_{p,a} + \left(|B|^{-1} \left[\sum_{k} |O_k| \rho_a(r_B/r_{0_k})^p \right]\right)^{1/p} |B|^{a/d}\|f\|^\sharp_{p,a},
 \end{align*}
where we have used that \mbox{$\sum_k |O_k| \simeq |B|$}. The proof is also achieved by invoking property \eqref{eq:covering2} to compute the remaining sum and estimate it by the expected quantity \mbox{$\rho_a(K_\phi) |B|^{a/d}\|f\|^\sharp_{p,a}$}.
\end{proof}

\subsection{Optimality of the logarithmic growth} We begin by observing that a measure preserving function that is Lipschitz is, in fact, bi-Lipschitz.

\begin{Lemm} \label{lem:Kphi} If \mbox{$\phi\in \LL$} preserves the Lebesgue measure then
$$ \log(K_\phi) \simeq \log\left(\sup_{x,y} \frac{|\phi(x)-\phi(y)|}{|x-y|}\right) \simeq \log\left(\sup_{x,y} \frac{|x-y|}{|\phi(x)-\phi(y)|}\right).$$
 \end{Lemm}

\begin{proof} By symmetry between \mbox{$\phi$} and \mbox{$\phi^{-1}$}, it remains for us to check that
\be{eq:a}  \log(K(\phi)) \lesssim \log\left(\sup_{x,y} \frac{|\phi(x)-\phi(y)|}{|x-y|}\right).\ee 
The map \mbox{$\phi$} is bi-Lipschitz and almost everywhere differentiable, with a derivative \mbox{$D\phi$} whose determinant is equal to the constant \mbox{$1$}. For each of these points \mbox{$x$}, 
 $$ \| D\phi(x)^{-1} \| \leq \max_{\lambda \in \Sigma} |\lambda|^{-1} \leq \left[\max_{\lambda\in \Sigma} |\lambda| \right]^{d-1} \leq \|D \phi(x)\|^{d-1},$$
 where \mbox{$\Sigma$} is the spectrum of \mbox{$\sqrt{D\phi(x)^* D \phi(x)}$}. In this way, we see that
 $$ \log(K(\phi)) \lesssim \log\left(\sup_{x,y} \frac{|\phi(x)-\phi(y)|}{|x-y|} + \left[\sup_{x,y} \frac{|\phi(x)-\phi(y)|}{|x-y|}\right]^{d-1}\right)$$
 which yields \eqref{eq:a}.
 \end{proof}
Our argument for the optimality of the logarithmic growth relies on properties of quasi-conformal mappings. Given that we can restrict our attention to bi-Lipschitz functions, we can take as our definition of \mbox{$\phi \in \LL$} being \mbox{$K$}-quasi-conformal that it satisfies the inequality
\[
\|D\phi(x)\|_{L^\infty}^d \leq K\det(D\phi)(x)
\] 
for all \mbox{$x \in \R^d$}. If \mbox{$\phi$} is also measure preserving, then the about inequality reduces to
\be{eq:qq}
\|D\phi\|_{L^\infty}^d \leq K,
\ee
so in this context, \mbox{$K$}-quasiconformality is simply that the size of the Lipschitz constant of \mbox{$\phi$} is bounded by \mbox{$K^{1/d}$}.
We will use the following theorem by H.M.~Riemann \cite[Thm.~3]{R}.
\begin{Theo}\label{reimann}
Assume that \mbox{$\phi \in \LL$} is orientation preserving. If the induced map \mbox{$f \mapsto f\circ\phi$} is a bijective isomorphism of BMO and satisfies
\be{eq:reimann}
\|f\circ\phi\|^\sharp_{p,0} \leq k\|f\|^\sharp_{p,0}
\ee
for all \mbox{$f \in \text{BMO}$}
then \mbox{$\phi$} is a \mbox{$K$}-quasiconformal mapping with \mbox{$K = e^{(d-1)(Ck-1)}$}, for a fixed \mbox{$C>0$} depending only on the dimension \mbox{$d$}.
\end{Theo}

To show the optimality of the logarithmic growth in Theorem \ref{thm}, take \mbox{$\phi \in \LL$} which is measure preserving and orientation preserving. We know that \mbox{$\phi$} is \mbox{$K$}-quasi-conformal and \eqref{eq:qq} holds. By Theorem \ref{thm} we have that there exists a constant \mbox{$k$} such that \eqref{eq:reimann} holds and so, from Theorem \ref{reimann}, \mbox{$K = e^{(d-1)(Ck-1)}$}. Combining these two facts, we see that
\[
\|D\phi\|_{L^\infty}^d \leq e^{(d-1)(Ck-1)}.
\]
Rearranging this and applying Lemma \ref{lem:Kphi} we see that
\[
\log(K_\phi) \lesssim k
\]
and so Theorem \ref{thm} gives the optimal behaviour of the constant in \mbox{$K_\phi$} when \mbox{$a=0$}.

\subsection{The behavior of some Carleson measures} \label{sec:carl}

Let \mbox{$\mu$} be a measure on \mbox{$\R^+ \times \R^d$} a Carleson measure : 
$$ \|\mu\|_{\mathcal C} := \sup_{\textrm{ball} B \subset \R^d}  \ |B|^{-1} \mu(T(B)) <\infty$$
where \mbox{$T(B)$} is the Carleson box over the ball \mbox{$B$} defined by
$$ T(B):= \left\{(x,t),\ x\in B,\ 0< t \leq r_B \right\}=B \times (0,r_B].$$

\begin{Defin} Let \mbox{$\mu$} be a Carleson measure and consider \mbox{$\phi$} a bi-Lipschitz measure-preserving map on \mbox{$\R^d$}. We denote \mbox{$\mu^{\sharp \phi}$} the pull-back measure, defined by
$$ \mu^ {\sharp \phi}(I\times A) = \mu(I \times \phi^ {-1}(A)),$$
for every time interval \mbox{$I$} and measurable set \mbox{$A\subset \R^d$}.  
\end{Defin}

\begin{Defin}  Let \mbox{$\beta$} be a measurable map from \mbox{$\R^+ \times \R^d\rightarrow \R$} and we define the measure 
$$\mu:=\mu_\beta = |\beta(t,x)|^2 \frac{dtdx}{t}.$$
Consequently, it comes
$$  d\mu^ {\sharp \phi}(t,x) = \left|\beta(t,\phi(x))\right|^2 \frac{dtdx}{t}$$
so \mbox{$\mu^ {\sharp \phi}= \mu_{\beta^\phi}$} with \mbox{$\beta^ \phi(t,x)=\beta(t,\phi(x)).$}
\end{Defin}

\begin{Theo} \label{thm:Carleson} Let \mbox{$\phi$} a bi-Lipschitz measure-preserving map on \mbox{$\R^d$} and \mbox{$\mu=\mu_\beta$} be a Carleson measure associated to some \mbox{$\beta\in L^\infty(\R^+ \times \R^d)$}. Then there exists an implicit constant (only dependent on \mbox{$n$}) such that \mbox{$\mu^{\sharp \phi}$} is a Carleson measure with
$$ \| \mu^ {\sharp \phi} \|_{\mathcal C} \lesssim \|\mu\|_{\mathcal C} + \log(K_\phi)  \|\beta\|_{L^\infty(\R^+ \times \R^d)}^2.$$
\end{Theo}

\begin{proof} Let consider \mbox{$B=B(x_0,r)$} a ball of \mbox{$\R^d$} and its Carleson box \mbox{$T(B)$}. We have to estimate 
\begin{align*}
 \mu^{\sharp \phi}(T(B)) & = \int_{[0,r]\times B } \left| \beta(t,\phi(x))\right|^2 \frac{dtdx}{t} \\
  & = \int_{[0,r]\times \phi(B) } \left| \beta(t,x)\right|^2 \frac{dtdx}{t}.
\end{align*}
Aiming that, we use the collection \mbox{$(O_k)$} given by Lemma \ref{lemma}  to cover \mbox{$\phi(B)$} (with \mbox{$p=1$}):
\begin{itemize}
 \item The collection \mbox{$(2O_k)_k$} is a bounded covering of \mbox{$\phi(B)$}
 \item The collection \mbox{$(O_k)_k$} is disjoint
 \item By writing \mbox{$r_{O_k}$} the radius of \mbox{$O_k$}, then for all \mbox{$p\in [1,\infty)$}
  \begin{equation} \frac{1}{|B|} \sum_k |O_k| \log(r_B/r_{O_k})   \lesssim \log(K_\phi). \label{eq:covering3} \end{equation}
\end{itemize}
Then (we remember that as previously we have \mbox{$r_{O_k} \leq r$}), it follows
\begin{align*}
 \mu^{\sharp \phi}(T(B)) & = \sum_k \int_{[0,r]\times 2O_k } \left| \beta(t,x)\right|^2 \frac{dtdx}{t} \\
 & \leq \sum_k \int_{[0,r_{O_k}]\times 2O_k } \left| \beta(t,x)\right|^2 \frac{dtdx}{t} + \sum_k \int_{[r_{O_k},r]\times 2O_k } \left| \beta(t,x)\right|^2 \frac{dtdx}{t} \\
 & \leq \sum_k \mu(T(2O_k)) + \|\beta\|_{L^\infty(\R^+ \times \R^d)}^2 \sum_k |O_k| \log(\frac{r}{r_{O_k}}) \\
 & \lesssim \|\mu\|_{\mathcal C} \left(\sum_k |O_k| \right) + \|\beta\|_{L^\infty(\R^+ \times \R^d)}^2 \log(K_\phi) |B| \\
 & \lesssim \left(\|\mu\|_{\mathcal C} + \log(K_\phi)  \|\beta\|_{L^\infty(\R^+ \times \R^d)}^2\right) |B|,
\end{align*}
where we used the doubling property of the Euclidean measure, the disjointness of the balls \mbox{$(O_k)_k$} and the property \eqref{eq:covering3}. The proof is also concluded.
\end{proof}

\begin{Coro} Let define \mbox{$SC$} the class of Carleson measure \mbox{$d\mu= |\beta(t,x)|^2 \frac{dtdx}{t}$} satisfying \mbox{$\|\beta\|_{L^\infty} \lesssim \|\mu\|_{\mathcal C}$}, equipped with the norm \mbox{$\|\mu\|_{\mathcal SC}:=\|\mu\|_{\mathcal C}$}. Then we have
$$ \| \mu^ {\sharp \phi} \|_{\mathcal SC} \lesssim \log(K_\phi) \|\mu\|_{\mathcal SC}.$$ 
\end{Coro}

\begin{ex}
We know that for some standard ``approximations of unity'' kernels \mbox{$(K_t)$}, for a \mbox{$L^1_{loc}$} function \mbox{$g$} we can build the measure
$$ d\mu_{g}(t,x) = \left|\int K_t(x,y) g(y) dy \right|^2 \frac{dtdx}{t}.$$
Then it is well-known that \mbox{$d\mu_g$} is a Carleson measure if and only if \mbox{$g\in BMO$} (see \cite{FS}). Moreover, it is easy to check that such measures belong to \mbox{${\mathcal SC}$}.
\end{ex}
\section{Applications to some PDEs}

\subsection{The Transport equation by a free-divergence vector field}

Let \mbox{$v:\R^d \rightarrow \R^d$} be a divergence free\footnote{ That means  $\nabla.v=0$.} Lipschitz vector field and  consider the transport equation:
\begin{equation} \left\{ \begin{array}{ll}
            \partial_t u - v\cdot \nabla (u) =0 \\
             u_{|t=0}=u_0, 
           \end{array} \right. \label{eq:transport} \end{equation}
with an initial data \mbox{$u_0$}.
Then it is well-known that a smooth solution is constant along the characteristics given by the vector field. Indeed, consider the flow \mbox{$\phi:\R_+\times\R^d  \rightarrow \R^d$}, solution of
$$ \left\{ \begin{array}{ll}
            \partial_t \phi = v(\phi) \\
             \phi(0,x)=x, 
           \end{array} \right.$$
then the  divergence free assumption on   \mbox{$v$} yields that  \mbox{$\phi(t,\cdot)$} is a  Lebesgue measure preserving diffeomorphism, for every \mbox{$t\in\mathbb R$}. Moreover, any smooth solution \mbox{$u$} of the transport equation is unique and is given by
$$ u(t,x) = u_0(\phi_t^{-1}(x)).$$

It is well-known by using Gronwall Lemma that 
$$ K(\phi(t,\cdot)) \lesssim e^{\|v\|_{\textrm{Lip}} t},$$
where \mbox{$\|v\|_{\textrm{Lip}}$} is the Lipschitz constant of the vector field.

As a consequence, the previous Theorem and Corollary imply the following

\begin{Theo} \label{thm-2} Let \mbox{$u$} be the unique solution of \eqref{eq:transport}. 
\begin{enumerate}
\item If \mbox{$u_0\in \text{BMO}$}, then \mbox{$u\in L^\infty_{loc}(BMO)$} and
$$
 \|u(t)\|_{\text{BMO}} \lesssim \left[1+\|v\|_{\textrm{Lip}} t\right] \|u_0\|_{\text{BMO}},\qquad \forall\, t\geq 0.
$$
\item If \mbox{$u_0\in \textrm{Lip}_p(a)$} (for some \mbox{$a\in(0,1]$} and \mbox{$p\in(1,\infty)$}), then \mbox{$u\in L^\infty_{loc}(\textrm{Lip}_p(a))$} and
$$
 \|u(t)\|_{\textrm{Lip}_p(a)} \lesssim e^{a\|v\|_{\textrm{Lip}} t} \|u_0\|_{\textrm{Lip}_p(a)},\qquad \forall\, t\geq 0.
 $$
\end{enumerate}
\end{Theo} 

\subsection{The perturbed transport  equation}
Consider the following transport equation with a linear Riesz-type second member term 
\begin{equation} 
\left\{ \begin{array}{ll}
            \partial_t \omega + (u\cdot \nabla) \omega = {\mathcal R} \omega \\
  \omega_{|t=0}=\omega_0, 
           \end{array} \right.
            \label{eq:perturbed-transport}
             \end{equation}
           where \mbox{$\mathcal R$} is a Riesz operator. This type of equation naturally arises when one considers for example
 the perturbed 2D Euler equations which is obtained by adding a zero order term to the incompressible 2D Euler system. Let us for example consider the following system
\begin{equation} 
\left\{ \begin{array}{ll}
            \partial_t u + (u\cdot \nabla)u = -\nabla p + Au \\
             \textrm{div}(u) =0 \\
            u_{|t=0}=u_0, 
           \end{array} \right. \label{eq:euler}
            \end{equation}
           with \mbox{$Au=(u^1,0)$}. Then the vorticity \mbox{$\omega:=\textrm{curl} (u) = \partial_1 u^2-\partial_2 u^1$} satisfies the following equation
 $$
\partial_t\omega+u\cdot\nabla\omega=\partial_{22}\Delta^{-1}\omega.
$$

The continuity of Riez operator on \mbox{$L^p$} for every \mbox{$1<p<\infty$}, the divergence free condition and Gronwall inequality imply together
\begin{equation}
\label{lp}
\|\omega(t)\|_{L^p}\leq \|\omega_0\|_{L^p}e^{C_pt},\qquad \forall t\geq 0.
\end{equation}
Here \mbox{$C_p=\|\mathcal R\|_{\mathcal L(L^p, L^p)}\simeq \frac{p^2}{p-1}$}. 
However, it is not clear at all how one can obtain an \mbox{$L^\infty$} estimates since the Riesz operator  \mbox{$\partial_{22}\Delta^{-1}$} is  not continuous on that space.
A natural idea is to replace the space \mbox{$L^\infty$}  by another space with similar "scaling" but stable for \mbox{$\mathcal R$} (such that \mbox{$BMO$} for example). However, in this case a problem of composition arises: an extra term depending of the Lipschitz norm of \mbox{$u$} appears and the estimate is no longer closable.

Theorem \ref{reimann} shows that we cannot avoid the constants generated by the composition with the flow (nor even to improve them). 
A bound for the \mbox{$BMO$}-norm similar to \eqref{lp} cannot be obtained directly: indeed with the best constants, we already have a quadratic estimate which is not {\it Gronwallisable}. However, Theorem \ref{thm} can be applied in order to get sharper {\it a priori} estimates. In fact, consider \mbox{$u$} be a smooth solution of \eqref{eq:perturbed-transport} and the corresponding flow \mbox{$\phi:\R_+\times\R^d \rightarrow \R^d$}, solution of
$$ 
\left\{ \begin{array}{ll}
 \partial_t \phi = u(t,\phi) \\
 \phi(0,x)=x. 
 \end{array} \right.
 $$
 Hence by Gronwall Lemma and using the $BMO$-boundedness of the Riesz transform, it comes
$$ 
\|\omega(t,\phi(t,\cdot))\|_{BMO} \lesssim \|\omega_0 \|_{BMO} \exp{(ct)},
$$
for some numerical constant \mbox{$c>0$}, and so by Theorem \ref{thm}
$$
 \|\omega(t,\cdot)\|_{BMO} \lesssim \|\omega_0 \|_{BMO}(1+ \|u\|_{L^1_t\textrm{Lip}}) \exp{(ct)}.
$$
If the vector-field satisfies
$$
 \|u\|_{L^1_t\textrm{Lip}}\leq \exp({\alpha t}),
 $$
 for some \mbox{$\alpha>0$}, then one has a similar estimate than all the \mbox{$L^p$} norm.
It is worthy of noticing that a rough estimate (involving \mbox{$K(\phi)$} instead of \mbox{$\log K(\phi)$}) gives 
$$
 \|\omega(t,\cdot)\|_{BMO} \lesssim \|\omega_0 \|_{BMO}\exp( \|u\|_{L^1_t\textrm{Lip}})  \exp{(ct)}.
$$
The  merit of the next result is only the improvement of the estimate of the growth of the \mbox{$BMO$} norm of the solution.
\begin{Prop} \label{thm-3} Let \mbox{$u$} be a divergence free vector fields and \mbox{$\omega$} a smooth solution of \eqref{eq:perturbed-transport}. If \mbox{$\omega_0\in \text{BMO}$}, then
$$ \|\omega(t)\|_{\text{BMO}} \lesssim \|\omega_0 \|_{BMO}(1+ \|u\|_{L^1_t\textrm{Lip}}) \exp{(ct)},
$$
for all $t\geq 0$.
\end{Prop} 

{\bf Acknowledgment: } The authors are thankful to Carlos Perez for having given the reference \cite{R} to prove the optimality of the logarithmic growth and to David Rule for valuable advices to improve this paper.

\end{document}